\documentclass[reqno]{amsart}
\usepackage[english]{babel}
\usepackage{amssymb,amsthm,amsfonts,amstext}
\usepackage{amsmath}
\usepackage{helvet} % selects Helvetica as sans-serif font
\usepackage{courier} % selects Courier as typewriter font
\usepackage{enumerate}

\newcounter{thmglobal}
\newtheorem{theorem}[thmglobal]{Theorem}

\newtheorem{lemma}{Lemma}
\newtheorem{proposition}{Proposition}

\newcommand{\<}{\langle}
\renewcommand{\>}{\rangle}
\renewcommand{\(}{\left(}
\renewcommand{\)}{\right)}

\newcommand{\splitalign}[2]{\vphantom{#2} #1 \right. \\ & \left. \vphantom{#1} #2}
\newcommand{\triplenorm}[1]{{\left\lvert\mkern-1.0mu\left\lvert\mkern-1.0mu\left\lvert #1 \right\rvert\mkern-1.0mu\right\rvert\mkern-1.0mu\right\rvert}}

\begin{document}

\title[Duality between Eigenelements of Ruelle and Koopman operators]{Duality between Eigenfunctions and Eigendistributions of Ruelle and Koopman operators via an integral kernel}

\author{P. Giulietti*}
\address{UFRGS, Instituto de Matem\'atica, Av. Bento Gon\c calves, 9500. CEP 91509-900, Porto Alegre, RS, Brasil}
%\curraddr{}
\email{paologiulietti.math@gmail.com}
\thanks{* Partially supported by CNPq-Brazil project number 407129/2013-8}

\author{A. O. Lopes**}
\address{UFRGS, Instituto de Matem\'atica, Av. Bento Gon\c calves, 9500. CEP 91509-900, Porto Alegre, RS, Brasil}
%\curraddr{}
\email{arturoscar.lopes@gmail.com}
\thanks{** Partially supported by CNPq, INCT}

\author{V. Pit***}
\address{UFRJ, Instituto de Matem\'atica, Av. Athos da Silveira Ramos 149, Centro de Tecnologia - Bloco C - Cidade Universit\'aria - Ilha do Fund\~ao, Rio de Janeiro, RJ, Brasil}
%\curraddr{}
\email{pit@im.ufrj.br}
\thanks{*** Partially supported by FAPESP project number 2011/12338-0 and CAPES}

\date{\today}

\begin{abstract}

We consider the classical dynamics given by a one sided shift on the Bernoulli space of $d$ symbols.
We study, on the space of H\"older functions, the eigendistributions of the Ruelle operator with a given potential.
Our main theorem shows that for any isolated eigenvalue, the eigendistributions of such Ruelle operator are dual to eigenvectors of a Ruelle operator with a conjugate potential.
We also show that the eigenfunctions and eigendistributions of the Koopman operator satisfy a similar relationship.
To show such results we employ an integral kernel technique, where the kernel used is the involution kernel.

\end{abstract}

\maketitle

\section{Introduction and main results}

Let $M = \{ 1, 2, \hdots, d \}$ be the classical alphabet with $d$ symbols.
We consider on $M$ the discrete distance $\tilde{d}$ in such way the distance among different points is equal to $1$.
The space $\Omega = M^\mathbb{N}$ of all the sequences $x = (x_1 x_2 \hdots), x_j \in M, j \in \mathbb{N}$ is equipped with the usual shift operator $\sigma : \Omega \to \Omega$ such that $\sigma(x_1 x_2 \hdots) = (x_2 x_3 \hdots)$ and with the distance :
\[ d_\Omega(x, x') = \sum_{n \ge 1} \frac{1}{2^n} \tilde{d}(x_n , x_n') \]
that makes $(\Omega, d_\Omega)$ into a compact metric space.
The choice of the exponential ratio $\frac{1}{2}$ is arbitrary, any value between zero and one can equally be chosen.
If $x, x' \in \Omega$, we will note $x \sim_n x'$ when their symbols agree until the coordinate $n \ge 1$.

Now take $\Omega^\star = M^\mathbb{N}$ another copy of the same space where points are written, for notational convenience, backwards $y = (\hdots y_2 y_1)$, and on which the shift $\sigma^\star : \Omega^\star \to \Omega^\star$ acts by $\sigma^\star(\hdots y_2 y_1) = (\hdots y_3 y_2)$.
We equip this space with the metric $d_{\Omega^\star}$ that makes it into a compact metric space, as well as with the relations $\sim_n$.

$M^\mathbb{Z}$ will be identified with the product $X = \Omega^\star \times \Omega$ where points are written $(y | x)$, $y \in \Omega^\star$ and $x \in \Omega$.
$X$ is equipped with its own two-sided shift $\hat{\sigma} : X \to X$ defined by :
\[ \hat{\sigma}(\hdots y_2 y_1 | x_1 x_2 \hdots) = (\hdots y_2 y_1 x_1 | x_2 x_3 \hdots) \]
and with the distance :
\[ d_X(u, u') = \sum_{n \in \mathbb{Z} \setminus 0} \frac{1}{2^{|n|}} \tilde{d}(u_n , v_n) = d_{\Omega^\star}(y, y') + d_\Omega(x, x') \]
where $u = (y | x)$, $u' = (y' | x')$ and :
\[ u_n = \begin{cases} x_n & \text{if $n > 0$} \\ y_{-n} & \text{if $n < 0$} \end{cases} \]
We denote by $\tau_y(x) = \tau_{y_1}(x) = (y_1 x)$ the inverse branch of $\sigma$ parameterized by $y \in \Omega^\star$, so that $\hat{\sigma}^{-1}(y | x) = (\sigma^\star(y) | \tau_y(x))$.

A map $f : \Omega \to \mathbb{C}$ is {\em $\theta$-H\"older} whenever there exists a $C \ge 0$ such that :
\[ \forall x, x' \in \Omega, \forall n \ge 0, x \sim_n x' \Rightarrow |f(x) - f(x')| \le C \theta^n \]
When $f$ is $\theta$-H\"older, we can define its seminorm $\| f \|_\theta$ by :
\[ \| f \|_\theta = \sup_{n \ge 1} \sup_{x \sim_n x'} \frac{|f(x) - f(x')|}{\theta^n} \]
By compactness of $\Omega$, all $\theta$-H\"older maps are bounded.
We denote by $\| \cdot \|_\infty$ the usual supremum norm, so that $\| \cdot \|_\infty + \| \cdot \|_\theta$ is a norm for the space of $\theta$-H\"older maps.
An easy computation shows that if $f$ and $g$ are two $\theta$-H\"older maps, then so do $f g$ and $\exp f$, and we have :
\begin{align*}
 \| f g \|_\theta    &\le \| f \|_\infty \| g \|_\theta + \| f \|_\theta \| g \|_\infty \\
 \| \exp f \|_\theta &\le \| f \|_\theta \exp \| f \|_\infty
\end{align*}
Note that if $\theta \le \frac{1}{2}$ this definition is equivalent to the usual definition of H\"older (or even Lipschitz) functions in the metric space $(\Omega, d_\Omega)$, but has the advantage that the exponent does not depend on the choice of the metric.
This is also the definition adopted in \cite{PP}.

Likewise, a map $g : X \to \mathbb{C}$ will be said to be {\em $\theta$-H\"older} whenever there exists a $C \ge 0$ such that :
\[ \forall u = (y | x), u' = (y' | x') \in X, \forall n \ge 0, x \sim_n x' \text{~and~} y \sim_n y' \Rightarrow |g(u) - g(u')| \le C \theta^n \]
and the smallest constant $C$ that satisfies this relation will be noted $\| g \|_\theta$.
It is trivial to check that a function is $\theta$-H\"older if and only if its partial maps $g(\cdot | x')$ and $g(y' | \cdot)$ are $\theta$-H\"older uniformly in $x' \in \Omega$ and $y' \in \Omega^\star$.

We will denote by $H_\theta(\Omega)$, $H_\theta(\Omega^\star)$ and $H_\theta(X)$ the spaces of H\"older functions on respectively $\Omega$, $\Omega^\star$ and $X$, equipped with their respective $\| \cdot \|_\infty + \| \cdot \|_\theta$ norms.
We will also note $\mathcal{C}(\Omega)$ the space of continuous maps on $\Omega$, equipped with the supremum norm.

Consider a function $A \in H_\theta(\Omega)$, that we will usually call a potential.
By adding a coboundary over $(X, \hat{\sigma})$ to $A$, it is possible to define a dual potential of $A$ that only depends on the past coordinates.
More precisely, we will prove in proposition \ref{prop:invkern} that there exist maps $W : X \to \mathbb{C}$, $W \in H_\theta(X)$ and $A^\star : \Omega^\star \to \mathbb{C}$, $A^\star \in H_\theta(\Omega^\star)$ such that for any $(y | x) \in X$ :
\begin{equation} \label{W0}
 A^\star(y) = A(\hat{\sigma}^{-1}(y | x)) + W(\hat{\sigma}^{-1}(y | x)) - W(y | x)
\end{equation}
where $A^\star$ does not depend on $x$.
Such a map $A^\star$ is called a {\em dual potential} of $A$, and any map $W : X \to \mathbb{C}$ that satisfies the relation \eqref{W0} is called an {\em involution kernel} between $A$ and $A^\star$.

For example, if the potential $A$ depends on a finite number of coordinates (that is, if there exists $k > 1$ such that $A(x_1 \hdots x_k, x_{k+1} \hdots) = A(x_1 \hdots x_k)$), then it easy to see that $A$ is $\theta$-H\"older for any $\theta < 1$, hence admits an involution kernel $W$ with the same regularity.
Moreover, when $A$ depends on the first two coordinates (i.e. $k = 2$), then section 5 from \cite{BLLbook} tells us that $A^\star$ is the transpose of the matrix whose entries are $A(x_1, x_2)$.
If $\Omega$ is seen as the coding of a smooth uniformly expanding dynamical system (like $T(x) = 2x \text{\,mod\,} 1$ on the circle), then the pull-back to $\Omega$ of any smooth potential will also be H\"older.

Let $\nu$ be any Borel probability measure on $M$, given by $\nu = \sum_{j=1}^d a_j \delta_j$ with $\sum_{j=1}^d a_j = 1$.
In order to simplify the notations, the integration of a function $f : M \to \mathbb{C}$ with respect to this probability $\nu$ will be written $\int f(a) da$.

We can define the {\em Ruelle operator} $\mathcal{L}_A : H_\theta(\Omega) \to H_\theta(\Omega)$ associated with the potential $A$ and the measure $\nu$ by :
\[ \mathcal{L}_A \varphi(x) = \int_M e^{A(a x)} \varphi(a x) da \]
It is more common to define the Ruelle operator via the counting a priori ``measure'' $\nu = \sum_{j=1}^d \delta_j$ and not using the probability $d a = \nu = \sum_{j=1}^d a_j \delta_j$.

For simplicity we write $A(a x)$ in the sense of $A(a x) = A(\tau_a(x))$, that is $(a x)$ is the element obtained by adding one extra symbol $a$ to $x$.
Along the paper we will not comment any further on the matter, since the context will allow the reader to understand what we are doing.
An exposition of the general theory of this operator, as well as and more references, can be found in \cite{BCLMS,LMMS}.

If $A^\star$ is a dual potential of $A$, we can also define the Ruelle operator
$\mathcal{L}_{A^\star} : H_\theta(\Omega^\star) \to H_\theta(\Omega^\star)$ associated to $A^\star$ by :
\[ \mathcal{L}_{A^\star} \varphi(y) = \int_M e^{A^\star(y a)} \varphi(y a) da \]

Note that the definition of these operators depends on the choice of the a priori measure $\nu$ on $M$.
The classical Ruelle operator $\sum_{t=1}^d e^{A(t x)} f(t x)$ which appears in the literature corresponds to our setting equipped with the a priori measure $\nu = \sum_{j=1}^d \delta_j$ (see \cite{PP} and \cite{Ruelle}).
This implies that the results presented in this article can be easily adapted to the usual case.
A more thorough discussion about this point can be found in \cite{LMMS}.

Recall that given $A \in H_\theta(\Omega)$ we denote its Birkhoff sums by :
\begin{equation} \label{eq:bir-sum}
 A^n(x) = \sum_{k=0}^{n-1} A(\sigma^k(x))
\end{equation}
Since $A$ is bounded (hence $e^{\Re A}$ as well), the Ruelle operator $\mathcal{L}_A$ has a finite spectral radius when acting on $H_\theta(\Omega)$.
According to theorem 1.5 from \cite{Bal}, the number $\rho$ defined as the limit :
\begin{equation} \label{eq:rho}
 \rho = \lim_{n \to \infty} \sup_{x' \in \Omega} \( \int dx_1 \hdots \int dx_n e^{\Re A^n(x_1 \hdots x_n x')} \)^{\frac{1}{n}}
\end{equation}
is an upper bound of this spectral radius.
Likewise, the spectral radius of $\mathcal{L}_{A^\star}$ admits an upper bound $\rho^\star$ defined similarly.

We shall be interested in the ``eigenelements'' of these Ruelle operators.
A map $\psi \in H_\theta(\Omega)$ is an {\em eigenfunction} of $\mathcal{L}_A$ associated to the eigenvalue $\lambda$ when :
\[ \forall x \in \Omega, \mathcal{L}_A \psi(x) = \lambda \psi(x) \]
Likewise, a continuous linear functional $\mathcal{D} : H_\theta(\Omega^\star) \to \mathbb{C}$ is an {\em eigendistribution} of $\mathcal{L}_{A^\star}$ associated with the eigenvalue $\lambda$ when :
\[ \forall \varphi \in H_\theta(\Omega^\star), \< \mathcal{D}, \mathcal{L}_{A^\star} \varphi \> = \left\< \mathcal{D}, y \mapsto \int_M e^{A^\star(y a)} \varphi(y a) da \right\> = \lambda \< \mathcal{D}, \varphi \> \]
We point out that we do not require $\lambda$ to be the main eigenvalue, hence the $\psi$ we consider might not be the main eigenfunction and $\mathcal{D}$ is not necessarily a measure.\footnote{We will use the two notations $y \mapsto \int_M e^{A^\star(y a)} \varphi(y a) da$ and $\int_M e^{A^\star(\cdot a)} \varphi(\cdot a) da$ indistinctly to indicate the needed function.}

Our main purpose in this article is to relate explicitly the eigendistributions $\mathcal{D}$ of the Ruelle operator associated with $A^\star$ and the eigenfunctions $\psi$ of the Ruelle operator associated with $A$.
More precisely, we will show that :

\begin{theorem} \label{th:ruelle}
 Let $A \in H_\theta(\Omega)$, $A^\star \in H_\theta(\Omega^\star)$ a dual potential of $A$, and $W \in H_\theta(X)$ an involution kernel between them.
 Denote by $\rho$ the upper bound of the spectral radius of $\mathcal{L}_A$ given by \eqref{eq:rho}.
 Then if $|\lambda| > \rho \theta$ the map :
 \[ \Phi_W : \mathcal{D} \in H_\theta(\Omega^\star)' \mapsto \left( x \in \Omega \mapsto \left\< \mathcal{D}, e^{W(\cdot | x)} \right\> \right) \in \mathcal{C}(\Omega) \]
 is a continuous linear operator that induces an isomorphism from the space of eigendistributions of $\mathcal{L}_{A^\star}$ for the eigenvalue $\lambda$ onto the space of eigenfunctions of $\mathcal{L}_A$ for the same eigenvalue $\lambda$.

 Moreover, there is an explicit expression of its inverse as the limit of a sequence of measures.
 Indeed, if $\psi \in H_\theta(\Omega)$ is a $\lambda$-eigenfunction of $\mathcal{L}_A$, then :
\begin{align*}
 \forall \varphi \in H_\theta(\Omega^\star), \< \mathcal{D}, \varphi \> = \lim_{n \to \infty} &\int dx_1 \hdots \int dx_n \psi(x_n \hdots x_1 x') \\
 &\hspace{1em} \frac{e^{A^n(x_n \hdots x_1 x')}}{\lambda^n} e^{- W(0^\infty x_n \hdots x_1 | x')} \varphi(0^\infty x_n \hdots x_1)
\end{align*}
 for any choice of $x' \in \Omega$ and $0^\infty = (\hdots 0) \in \Omega^\star$, is an element of $H_\theta(\Omega^\star)'$, does not depend on $x'$, is a $\lambda$-eigendistribution of $\mathcal{L}_{A^\star}$, and satisfies $\Phi_W(\mathcal{D}) = \psi$.
\end{theorem}

Since we do not require $\lambda$ to be the leading eigenvalue, this theorem generalizes the results of \cite{BLT} which consider the analogous result just for the main eigenfunction of the Ruelle operator associated with $A$, in which case the eigendistribution $\mathcal{D}$ is a probability measure $\mu$.
In contrast, our approach can be applied to questions for which it is natural to analyze other eigenelements that are not associated with the leading eigenvalue (see for instance \cite{Fro}).

However, according to theorem 10.2 from \cite{PP} or theorem 1.5 from \cite{Bal}, the condition $|\lambda| > \rho \theta$ implies that, if $\lambda$ is an eigenvalue of $\mathcal{L}_A$, then it must lie in its isolated spectrum.
We point out that this condition is not empty.
Indeed, one can find in \cite{KR} an example of an analytic map $T$ of the unit circle, homeomorphic to $2x \text{\,mod\,} 1$, whose derivative is everywhere greater than $\frac{3}{2}$, and whose transfer operator $\mathcal{L}_{-\log|T'|}$ acting on the space of $\mathcal{C}^1$ maps admits an eigenvalue $\lambda$ such that $\frac{3}{4} < |\lambda| < 1$.
When pulled-back to the fulled shift with two symbols, $- \log |T'|$ lifts to a $\frac{2}{3}$-H\"older potential and $\lambda$ is an isolated non-maximal eigenvalue for its associated Ruelle operator.

Distributions related to eigenfunctions and the involution kernel appeared in \cite{Pit1}, \cite{Pit2}, \cite{Pit3} and \cite{LT} (plus, in a non rigorous form, in \cite{Bo1} and \cite{Bo2}).
Moreover, distributions have been intensively studied, starting with \cite{BKL02}, through anisotropic Banach spaces and there is now a vast literature on them.
On such spaces of generalized functions, the transfer operator is quasi-compact and its properties can be used in attacking a variety of problems (see for example \cite{GLP13}, \cite{BT08}).
We hope that the approach presented here could be extended to such context.

In the last section, we will translate this result in the language of Koopman operators.
If $B \in H_\theta(\Omega)$, the {\em Koopman operator} $U_B : H_\theta(\Omega) \to H_\theta(\Omega)$ associated with the potential $B$ is defined as :
\[ \forall x \in \Omega, U_B \varphi (x) = e^{B(x)} \varphi \circ \sigma(x) \]
One can similarly define a Koopman operator $U_{B^\star} : H_\theta(\Omega^\star) \to H_\theta(\Omega^\star)$ associated with a potential $B^\star \in H_\theta(\Omega^\star)$.
We will prove the following result :

\begin{theorem} \label{th:koopman}
 Let $B \in H_\theta(\Omega)$ and $C^\star \in H_\theta(\Omega^\star)$, which is not necessarily cohomologous to $B$.
 Let $A \in H_\theta(\Omega)$, and $A^\star \in H_\theta(\Omega^\star), W \in H_\theta(X)$ the associated dual potential and involution kernel.
 Denote by $\rho$ the upper bound of the spectral radius of $\mathcal{L}_A$ given by \eqref{eq:rho}.
 Then there exist $f \in H_\theta(\Omega), \alpha > 0$ depending on $A$ and $B$, $g \in H_\theta(\Omega^\star), \beta > 0$ depending on $A^\star$ and $C^\star$, such that $f, g > 0$ and for every $\lambda$ satisfying $|\lambda| > \rho \theta$ the map :
 \[ \Psi_{A,B,C^\star} : \nu \in H_\theta(\Omega^\star)' \mapsto f \Phi_W \( \frac{1}{g} \nu \) \in H_\theta(\Omega) \]
 is a continuous linear operator that induces an isomorphism from the space of eigendistributions of $U_{C^\star}$ for the eigenvalue $\frac{\beta}{\lambda}$ onto the space of eigenfunctions of $U_B$ for the eigenvalue $\frac{\alpha}{\lambda}$.

 Moreover, there is an explicit expression of its inverse as the limit of a sequence of measures.
 Indeed, if $\psi \in H_\theta(\Omega)$ is a $\frac{\alpha}{\lambda}$-eigenfunction of $U_B$, then :
 \begin{align*}
  \forall \varphi \in H_\theta(\Omega^\star), \< \mathcal{D}, \varphi \> = \lim_{n \to \infty} &\int dx_1 \hdots \int dx_n \( \frac{\psi}{f} \)(x_n \hdots x_1 x') \\
  &\hspace{1em} \frac{e^{A^n(x_n \hdots x_1 x')}}{\lambda^n} e^{- W(0^\infty x_n \hdots x_1 | x')} (g \varphi)(0^\infty x_n \hdots x_1)
 \end{align*}
 for any choice of $x' \in \Omega$ and $0^\infty = (\hdots 0) \in \Omega^\star$, is an element of $H_\theta(\Omega^\star)'$, does not depend on $x'$, is a $\frac{\beta}{\lambda}$-eigendistribution of $U_{C^\star}$, and satisfies $\Psi_{A,B,C^\star}(\mathcal{D}) = \psi$.
\end{theorem}

We recall that the product $f \nu \in H_\theta(\Omega^\star)'$ of a map $f \in H_\theta(\Omega^\star)$ and a distribution $\nu \in H_\theta(\Omega^\star)'$ is the distribution defined by :
\[ \forall \varphi \in H_\theta(\Omega^\star), \< f \nu, \varphi \> = \< \nu, f \varphi \> \]

Note that $B$ and $C^\star$ can be a priori chosen independently, but that the corresponding points of the spectra of $U_B$ and $U_{C^\star}$ for which the theorem is meaningful are constrained by the values of $\alpha$ and $\beta$, and by the condition $|\lambda| > \rho \theta$.

This result is related in some sense to the non rigorous reasoning of \cite{Bo1} and \cite{Bo2}.

One of the motivations for the material of the last section comes from \cite{ABL}, where relations between the spectral radius of the Ruelle operator associated with some potential and the spectral radius of a weighted shift operator (the analogue of our Koopman operator) associated with another potential are established.
This can be related to our proposition \ref{prop:trans-koop-func}.

\subsection*{Acknowledgements}

The authors would like to thank the referee for pointing out a few missprints in the original manuscript.

\section{The involution kernel $W$ and the map $\Phi_W$}

Let $A \in H_\theta(X)$ with $\theta < 1$.
We shall first prove the existence of a $\theta$-H\"older dual potential $A^\star$ of $A$ and of a $\theta$-H\"older involution kernel $W$ between them, as claimed in the introduction.

\begin{proposition} \label{prop:invkern}
If $A \in H_\theta(\Omega)$ then there exist $A^\star \in H_\theta(\Omega^\star)$ and $W \in H_\theta(X)$ such that for every $(y | x) \in X$ :
\[ A^\star(y) = A(\tau_y(x)) + W(\hat{\sigma}^{-1}(y | x)) - W(y | x) \]
does not depend on $x$.
Moreover, one can choose $W$ and $A^\star$ in such a way that :
\[ \| W \|_\theta \le \| A \|_\theta \frac{3 \theta}{1 - \theta} \hspace{2em}\text{~and~}\hspace{2em} \| A^\star \|_\theta \le \| A \|_\theta \frac{2}{1 - \theta} \]
\end{proposition}

\begin{proof}
We will use the Sina\"i method described in \cite{BLT}.
Fix $z \in \Omega$.
For any $y \in \Omega^\star$, $x \in \Omega$ and $n \ge 1$, let :
\[ W_n(y | x) = \sum_{k=1}^n A(y_k \hdots y_1 x) - A(y_k \hdots y_1 z) \]
Since $A$ is $\theta$-H\"older, we have :
\[ \sum_{k=1}^n \| A(y_k \hdots y_1 x) - A(y_k \hdots y_1 z) \|_\infty \le \| A \|_\theta \sum_{k=1}^n \theta^k = \| A \|_\theta \theta \frac{1 - \theta^n}{1 - \theta} \]
hence this series converges uniformly to :
\[ W(y | x) = \sum_{k \ge 1} A(y_k \hdots y_1 x) - A(y_k \hdots y_1 z) \]

We will now show that $W \in H_\theta(X)$.
First, if $x, x' \in \Omega$ are such that $x \sim_n x'$, then for every $y \in \Omega^\star$ and $p \ge 1$ we have :
\begin{align*}
 | W_p(y | x) - W_p(y | x') | &= \left| \sum_{k=1}^p A(y_k \hdots y_1 x) - A(y_k \hdots y_1 x') \right| \\
 &\le \| A \|_\theta \sum_{k=1}^p \theta^{k+n} = \| A \|_\theta \theta^{n+1} \frac{1 - \theta^p}{1 - \theta}
\end{align*}
which gives after taking $p \to \infty$ :
\[ | W(y | x) - W(y | x') | \le \| A \|_\theta \frac{\theta}{1 - \theta} \theta^n \]
Now, if $y, y' \in \Omega^\star$ are such that $y \sim_n y'$, then for every $x \in \Omega$ and $p \ge n$ we have :
\begin{align*}
 | W_p(y | x) &- W_p(y' | x) | \\
 &= \left| \sum_{k=1}^p (A(y_k \hdots y_1 x) - A(y_k' \hdots y_1' x)) - (A(y_k \hdots y_1 z) - A(y_k' \hdots y_1' z)) \right| \\
 &\le \sum_{k=n+1}^p | A(y_k \hdots y_1 x) - A(y_k \hdots y_1 z) | + | A(y_k' \hdots y_1' x) - A(y_k' \hdots y_1' z) | \\
 &\le 2 \| A \|_\theta \sum_{k=n+1}^p \theta^k = 2 \| A \|_\theta \theta^{n+1} \frac{1 - \theta^{p-n}}{1 - \theta}
\end{align*}
which gives after taking $p \to \infty$ :
\[ | W(y | x) - W(y' | x) | \le 2 \| A \|_\theta \frac{\theta}{1 - \theta} \theta^n \]
Hence we get from these two estimates that $W$ is $\theta$-H\"older with :
\[ \| W \|_\theta \le \| A \|_\theta \frac{3 \theta}{1 - \theta} \]

Observe that $W$ satisfies :
\begin{align*}
 W(\sigma^{-1}(y | x)) &- W(y | x) + A(\tau_y(x)) \\
 &= A(y_1 x) + \sum_{k \ge 1} A(y_{k+1} \hdots y_2 y_1 x) - A(y_{k+1} \hdots y_2 z) \\
 &\hspace{5em} - \sum_{k \ge 1} A(y_k \hdots y_1 x) - A(y_k \hdots y_1 z) \\
 &= \sum_{k \ge 1} A(y_k \hdots y_1 z) - A(y_{k+1} \hdots y_2 z)
\end{align*}
hence we have :
\begin{equation} \label{astar}
 W(\sigma^{-1}(y | x)) - W(y | x) + A(\tau_y(x)) = A(y_1 z) + \sum_{k \ge 2} A(y_k \hdots y_2 y_1 z) - A(y_k \hdots y_2 z)
\end{equation}
which only depends on $y$.
This lets us define $A^\star : \Omega^\star \to \mathbb{C}$ by :
\[ \forall y \in \Omega^\star, A^\star(y) = W(\sigma^{-1}(y | x)) - W(y | x) + A(\tau_y(x)) \]
for any choice of $x \in \Omega$.
This map satisfies \eqref{W0} by construction.

The only point remaining is to check that $A^\star \in H_\theta(\Omega^\star)$.
To this end, we will estimate the variations of $A^\star$ using \eqref{astar}.
Let $x \in \Omega$ and $y, y' \in \Omega^\star$ such that $y \sim_n y'$ with $n \ge 0$.
If $n = 0$, then this means that :
\begin{align*}
 |A^\star(y) &- A^\star(y')| \le |A(y_1 z) - A(y_1' z)| \\
 &\hspace{1em}+ \sum_{k \ge 2} | A(y_k \hdots y_2 y_1 z) - A(y_k \hdots y_2 z) | + | A(y_k' \hdots y_2' y_1' z) - A(y_k' \hdots y_2' z) | \\
 &\le \| A \|_\theta \theta^0 + 2 \| A \|_\theta \sum_{k \ge 2} \theta^{k-1} = \| A \|_\theta \( 1 + \frac{2 \theta}{1 - \theta} \) \theta^0 = \| A \|_\theta \frac{1 + \theta}{1 - \theta} \theta^0
\end{align*}
On the other hand, if $n \ge 1$, we have $\tau_y(z) = \tau_{y'}(z)$ and then :
\begin{align*}
 |A^\star(y) &- A^\star(y')| \\
 &\le \left| \sum_{k \ge 2} (A(y_k \hdots y_2 y_1 z) - A(y_k' \hdots y_2' y_1' z)) - (A(y_k \hdots y_2 z) - A(y_k' \hdots y_2' z)) \right| \\
 &\le \sum_{k \ge n+1} | A(y_k \hdots y_2 y_1 z) - A(y_k \hdots y_2 z) | + | A(y_k' \hdots y_2' y_1' z) - A(y_k' \hdots y_2' z) | \\
 &\le 2 \| A \|_\theta \sum_{k \ge n+1} \theta^{k-1} = \| A \|_\theta \frac{2}{1 - \theta} \theta^n
\end{align*}
Hence $A^\star$ is $\theta$-H\"older, and since $\theta < 1$ we have the estimate :
\[ \| A^\star \|_\theta \le \| A \|_\theta \frac{2}{1 - \theta} \tag*{\qedhere} \]
\end{proof}

Note that this result is stronger than the very similar proposition 1.2 from \cite{PP}, since they showed that $A^\star$ and $W$ are $\sqrt{\theta}$-H\"older in their setting while we get that they are $\theta$-H\"older.
This is due to the fact that our initial potential $A$ only depends on the future $\Omega$ and not on $X$ in its entirety.

From now on, we shall assume that $W$ and $A^\star$ are given by proposition \ref{prop:invkern} from $A$.

Since the partial functions $y \in \Omega^\star \mapsto e^{W(y | x)}$ are all $\theta$-H\"older for every $x \in \Omega$, it is clear that $\Phi_W$ is well-defined on $H_\theta(\Omega)'$.
We shall now prove that its image is made of continuous functions.
To this end, we will need a regularity result about the $\theta$-norm of these partial functions.

\begin{lemma}
Let $k \in H_\theta(X)$.
Then the map $f : x \in \Omega \mapsto \| k(\cdot | x) \|_\theta$ is well-defined and continuous.
\end{lemma}

\begin{proof}
Applying the definition of $\theta$-H\"older maps, it is clear that $f$ is well-defined, and even bounded by $\| k \|_\theta$.
Suppose that $f$ is not continuous at some $x' \in \Omega$, i.e. that there exists an $\varepsilon > 0$ such that for every $\alpha > 0$ there is a $x \in \Omega$ such that $d_\Omega(x, x') < \alpha$ and :
\[ |f(x) - f(x')| > \varepsilon \]
By taking $\alpha = \frac{1}{2^n}$, we get the existence of a sequence $x_n \in \Omega$ such that $x_n \sim_n x$ for every $n$ and :
\[ \forall n \ge 0, |f(x_n) - f(x')| > \varepsilon \]
Note that $x_n \to x$ when $x \to \infty$.
We can assume that there is an infinite subset $I$ of integers such that :
\[ \forall n \in I, f(x') > \varepsilon + f(x_n) \]
(the other case can be proven in a similar fashion).
Now since :
\[ f(x') = \sup_p \sup_{y_1 \sim_p y_2} \frac{|k(y_1 | x') - k(y_2 | x')|}{\theta^p} \]
there exists $q$ such that :
\[ \sup_{y_1 \sim_q y_2} \frac{|k(y_1 | x') - k(y_2 | x')|}{\theta^q} > f(x') - \frac{\varepsilon}{2} \]
But the map $(y_1, y_2) \mapsto k(y_1 | x') - k(y_2 | x')$ is continuous over :
\[ \Delta_q = \left\{ (y_1, y_2) \in \Omega^\star \times \Omega^\star \,\middle|\, y_1 \sim_q y_2 \right\} \]
which, as a closed subset of the compact $\Omega^\star \times \Omega^\star$, is compact ; so one can find $z_1 \sim_q z_2$ such that :
\[ \frac{|k(z_1 | x') - k(z_2 | x')|}{\theta^q} > f(x') - \frac{\varepsilon}{2} \]
On the other hand, we always have :
\[ f(x_n) = \sup_p \sup_{y_1 \sim_p y_2} \frac{|k(y_1 | x_n) - k(y_2 | x_n)|}{\theta^p} \ge \frac{|k(z_1 | x_n) - k(z_2 | x_n)|}{\theta^q} \]
so this gives :
\[ \forall n \in I, \frac{|k(z_1 | x') - k(z_2 | x')|}{\theta^q} > \frac{\varepsilon}{2} + \frac{|k(z_1 | x_n) - k(z_2 | x_n)|}{\theta^q} \]
This yields a contradiction when $n$ goes to infinity along $I$.
\end{proof}

\begin{lemma}
For every $\nu \in H_\theta(\Omega^\star)', \Phi_W(\nu) \in \mathcal{C}(\Omega)$.
\end{lemma}

\begin{proof}
Let $\psi = \Phi_W(\nu) : \Omega \to \mathbb{C}$.
Fix $\varepsilon > 0$ and $x_0 \in \Omega$.
For every $x \in \Omega$, we have :
\begin{align*}
 \left| \psi(x) - \psi(x_0) \right| &= \left| \left\< \nu, e^{W(\cdot | x)} - e^{W(\cdot | x_0)} \right\> \right| \\
 &\le \triplenorm{\nu} \( \left\| e^{W(\cdot | x)} - e^{W(\cdot | x_0)} \right\|_\infty + \left\| e^{W(\cdot | x)} - e^{W(\cdot | x_0)} \right\|_\theta \)
\end{align*}
where $\triplenorm{\nu}$ is the operator norm of the continuous linear functional $\nu : H_\theta(\Omega^\star) \to \mathbb{C}$.
Since $W$ is continuous, then so does :
\[ x \in \Omega \mapsto \left\| e^{W(\cdot | x)} - e^{W(\cdot | x_0)} \right\|_\infty \]
hence there exists $\alpha_1 > 0$ such that :
\[ d_\Omega(x, x_0) < \alpha_1 \Rightarrow \left\| e^{W(\cdot | x)} - e^{W(\cdot | x_0)} \right\|_\infty < \varepsilon \]
But, according to the previous lemma applied to $k(y | x) = e^{W(y | x)} - e^{W(y | x_0)}$, we know that there exists $\alpha_2 > 0$ such that :
\[ d_\Omega(x, x_0) < \alpha_2 \Rightarrow \left\| e^{W(\cdot | x)} - e^{W(\cdot | x_0)} \right\|_\theta < \varepsilon \]
This shows that $\psi$ is continuous at $x_0$.
\end{proof}

Finally, let us check that $\Phi_W$ is continuous.

\begin{lemma}
$\Phi_W : H_\theta(\Omega^\star)' \to \mathcal{C}(\Omega)$ is continuous for the operator norm $\triplenorm{\cdot}$ on $H_\theta(\Omega^\star)'$ and for the supremum norm $\| \cdot \|_\infty$ on $\mathcal{C}(\Omega)$.
\end{lemma}

\begin{proof}
For any $\nu \in H_\theta(\Omega^\star)'$, we have :
\[ \forall x \in \Omega, | \Phi_W(\nu)(x) | \le \triplenorm{\nu} \( \left\| e^{W(. | x)} \right\|_\infty + \left\| e^{W(. | x)} \right\|_\theta \) \le \( \left\| e^W \right\|_\infty + \left\| e^W \right\|_\theta \) \triplenorm{\nu} \]
using that the supremum norm and the $\theta$-H\"older seminorm of the partial maps of $e^W$ are bounded from above by the same quantities for $e^W : X \to \mathbb{C}$ itself.
This shows exactly that the linear operator $\Phi_W$ is continuous for the appropriate norms.
\end{proof}

\section{Ruelle operator duality}

In this section, we will prove that if $|\lambda| > \rho \theta$ then $\Phi_W$ is bijective from the space of $\lambda$-eigendistributions of $\mathcal{L}_{A^\star}$ to the space of $\lambda$-eigenfunctions of $\mathcal{L}_A$.

We will first need to relate the upper bound $\rho$ of the spectral radius of $\mathcal{L}_A$ on $H_\theta(\Omega)$ with the upper bound $\rho^\star$ of the spectral radius of $\mathcal{L}_{A^\star}$ on $H_\theta(\Omega^\star)$.

\begin{lemma} \label{rho}
$\rho = \rho^\star$.
\end{lemma}

\begin{proof}
For every $x' \in \Omega$, $y' \in \Omega^\star$ and $x_1, \hdots, x_n \in M$, by iterating \eqref{W0} and using the telescopic property of the Birkhoff sums $A^n$, we obtain :
\begin{equation} \label{eq:iteration}
 A^n(x_1 \hdots x_n x') = W(y' x_1 \hdots x_n | x') - W(y' | x_1 \hdots x_n x') + (A^\star)^n(y' x_1 \hdots x_n)
\end{equation}
Taking the absolute value of the exponential of this equality, and integrating over $x_1, \hdots, x_n$, we get :
\begin{align*}
 \int dx_1 & \hdots \int dx_n e^{\Re A^n(x_1 \hdots x_n x')} \\
 &= \int dx_1 \hdots \int dx_n e^{\Re (A^\star)^n(y' x_1 \hdots x_n)} e^{\Re W(y' x_1 \hdots x_n | x') - \Re W(y' | x_1 \hdots x_n x')}
\end{align*}
Note that $W \in H_\theta(X)$, thus $\Re W$ is uniformly bounded on $X$ and there are $A, B \in \mathbb{R}$ such that $A \le \Re W \le B$.
This gives :
\[ e^{A-B} \le \frac{\int dx_1 \hdots \int dx_n e^{\Re A^n(x_1 \hdots x_n x')}}{\int dx_1 \hdots \int dx_n e^{\Re (A^\star)^n(y' x_1 \hdots x_n)}} \le e^{B-A} \]
which implies, once we take the power $\frac{1}{n}$ of this inequality and let $n$ go to the infinity, that $\rho = \rho^\star$.
\end{proof}

Let $\lambda \in \mathbb{C}$ such that $|\lambda| > \rho \theta$.
In particular, if $\lambda$ is an eigenvalue of $\mathcal{L}_A$, then it must lie in its isolated spectrum.
Fix $\varepsilon > 0$ such that $|\lambda| > (\rho + \varepsilon) \theta$.
Thanks to lemma \ref{rho}, there exists an integer $n_0 \ge 0$ such that :
\begin{align*}
 \forall n \ge n_0, \forall x' \in \Omega,       &\int dx_1 \hdots \int dx_n e^{\Re A^n(x_1 \hdots x_n x')} \le (\rho + \varepsilon)^n \\
 \forall n \ge n_0, \forall y' \in \Omega^\star, &\int dy_1 \hdots \int dy_n e^{\Re (A^\star)^n(y' y_1 \hdots y_n)} \le (\rho + \varepsilon)^n
\end{align*}

We also take an eigenfunction $\psi \in H_\theta(\Omega)$ of $\mathcal{L}_A$ for this eigenvalue $\lambda$, which can be $0$ if the eigenspace is empty.
The goal of this section is to find a preimage of $\psi$ by $\Phi_W$.

We choose an arbitrary point $0$ in $M$, and let $0^\infty = (\hdots 0) \in \Omega^\star$.
$0^\infty$ is a fixed point for $\sigma^\star$, which will be used in the following as a reference point for our construction.
For every $n \ge 0$ and $x' \in \Omega$, we define a linear functional $\mathcal{D}_{n, x'} : H_\theta(\Omega^\star) \to \mathbb{C}$ using the involution kernel $W$ :
\begin{align*}
 \< \mathcal{D}_{n, x'}, \varphi \> = \int dx_1 \hdots \int dx_n \psi(x_n \hdots x_1 x') &\frac{e^{A^n(x_n \hdots x_1 x')}}{\lambda^n} \\
 &\hspace{1em} e^{- W(0^\infty x_n \hdots x_1 | x')} \varphi(0^\infty x_n \hdots x_1)
\end{align*}
When $n = 0$, this expression reduces to :
\[ \< \mathcal{D}_{0, x'}, \varphi \> = \psi(x') e^{-W(0^\infty | x')} \varphi(0^\infty) \]

The linear functionals $\mathcal{D}_{n, x'}$ are continuous, and even better since they actually are Radon measures as sums of Dirac deltas.
To help the reader the next lemmas have the objective to let us understand $\lim_{n \to \infty} \mathcal{D}_{n, x'}$ and, how, in the limit, we do not depend from the choice of $x'$.

\begin{lemma} \label{dncont}
For every $n \ge n_0$, $x' \in \Omega$ and $\varphi \in H_\theta(\Omega^\star)$ :
\[ | \< \mathcal{D}_{n, x'}, \varphi \> | \le \| \psi \|_\infty \left\| e^{-W} \right\|_\infty \( \frac{\rho + \varepsilon}{|\lambda|} \)^n \| \varphi \|_\infty \]
\end{lemma}

\begin{proof}
A direct majoration yields :
\[ | \< \mathcal{D}_{n, x'}, \varphi \> | \le \| \psi \|_\infty \left\| e^{-W} \right\|_\infty \| \varphi \|_\infty \frac{1}{|\lambda|^n} \int dx_1 \hdots \int dx_n e^{\Re A^n(x_n \hdots x_1 x')} \]
where the rightmost integral is bounded by $(\rho + \varepsilon)^n$.
\end{proof}

Note that this bound can get arbitrarily large as $n$ goes to infinity since one can have $|\lambda| < \rho$.

The action of the involution kernel with respect to $A$ and $A^\star$ allows to construct a recurrence relation between the measures $\mathcal{D}_{n, x'}$, which can be expressed in terms of $\mathcal{L}_{A^\star}$.

\begin{lemma} \label{int}
For every $n \ge 0$, $x' \in \Omega$, $\varphi \in H_\theta(\Omega^\star)$ :
\begin{equation} \label{eq1}
 \< \mathcal{D}_{n+1, x'}, \varphi \> = \frac{1}{\lambda} \int da \left\< \mathcal{D}_{n, a x'}, y \mapsto e^{A^\star(y a)} \varphi(y a) \right\>
\end{equation}
\end{lemma}

\begin{proof}
Note that for any $x' \in \Omega$ we have, due to equation \eqref{eq:bir-sum} :
\[ A^{n+1}(x_{n+1} \hdots x_1 x') = A^n(x_{n+1} \hdots x_1 x') + A(x_1 x') \]
Moreover, by equation \eqref{W0} we have that :
\begin{align*}
 A(x_1 x') &- W(0^\infty x_{n+1} \hdots x_1 | x') \\
 &= A^\star(0^\infty x_{n+1} \hdots x_1) - W(0^\infty x_{n+1} \hdots x_2 | x_1 x')
\end{align*}
Then by definition of $\mathcal{D}_{n+1, x'}$, we get for $n \ge 0$ :
\begin{align*}
 \< \mathcal{D}_{n+1, x'}, \varphi \> &= \int dx_1 \hdots \int d_{x_{n+1}} \psi(x_{n+1} \hdots x_1 x') \frac{e^{A^n(x_{n+1} \hdots x_1 x')}}{\lambda^{n+1}} \\
 &\hspace{5em} e^{A(x_1 x') - W(0^\infty x_{n+1} \hdots x_1 | x')} \varphi(0^\infty x_{n+1} \hdots x_1) \\
 &= \frac{1}{\lambda} \int dx_1 \hdots \int d_{x_{n+1}} \psi(x_{n+1} \hdots x_1 x') \frac{e^{A^n(x_{n+1} \hdots x_1 x')}}{\lambda^n} \\
 &\hspace{5em} e^{A^\star(0^\infty x_{n+1} \hdots x_1) - W(0^\infty x_{n+1} \hdots x_2 | x_1 x')} \varphi(0^\infty x_{n+1} \hdots x_1) \\
 &= \frac{1}{\lambda} \int dx_1 \left[ \splitalign{\int dx_2 \hdots \int d_{x_{n+1}} \psi(x_{n+1} \hdots x_2 x_1 x') \frac{e^{A^n(x_{n+1} \hdots x_2 x_1 x')}}{\lambda^n}}{\hspace{5em} e^{A^\star(0^\infty x_{n+1} \hdots x_2 x_1) - W(0^\infty x_{n+1} \hdots x_2 | x_1 x')} \varphi(0^\infty x_{n+1} \hdots x_2 x_1)} \right] \\
 &= \frac{1}{\lambda} \int dx_1 \left\< \mathcal{D}_{n, x_1 x'}, y \mapsto e^{A^\star(y x_1)} \varphi(y x_1) \right\> \tag*{\qedhere}
\end{align*}
\end{proof}

We stress that this lemma relates $\mathcal{D}_{n+1, x'}$ to $\mathcal{D}_{n, a x'}$, which have different base points.
In order to use this relation later to show that any accumulation point of $(\mathcal{D}_{n, x'})_{n \ge 0}$ is invariant for $\mathcal{L}_{A^\star}$, we will need to get rid of the dependence with respect to the base point.
To this end, we show that $| \< \mathcal{D}_{n, z_1'}, \varphi \> - \< \mathcal{D}_{n, z_2'}, \varphi \> |$ becomes exponentially small whenever $n$ goes to infinity.

\begin{lemma} \label{oo}
For every $n \ge n_0$, $z_1', z_2' \in \Omega$ and $\varphi \in H_\theta(\Omega)$ :
\[ | \< \mathcal{D}_{n, z_1'}, \varphi \> - \< \mathcal{D}_{n, z_2'}, \varphi \> | \le \left\| \psi(\cdot) e^{-W(0^\infty | \cdot)} \right\|_\theta \| \varphi \|_\infty \( \frac{(\rho + \varepsilon) \theta}{|\lambda|} \)^n \]
\end{lemma}

\begin{proof}
Recall that $x \mapsto \psi(x) e^{-W(0^\infty | x)}$ is $\theta$-H\"older as the product of two $\theta$-H\"older maps.
Fix $n \ge n_0$ and $z_1', z_2' \in \Omega$.
We use equation \eqref{eq:iteration} to produce a pair of relations between $A^n$ and ${A^\star}^n$ for the base points $z_i'$, $i = 1, 2$, that is :
\begin{align*}
 A^n(x_n \hdots x_1 z_i') &- W(0^\infty x_n \hdots x_1 | z_i') \\
 &= (A^\star)^n(0^\infty x_n \hdots x_1) - W(0^\infty | x_n \hdots x_1 z_i')
\end{align*}
Note that $x_1 \hdots x_n z_1' \sim_n x_1 \hdots x_n z_2'$ in $\Omega$ for every $x_1, \hdots, x_n \in M$ since the $n$-prefixes are the same.
Therefore, by definition of $\mathcal{D}_{n, z}$ and of $n_0$, we get :
\begin{align*}
 | \< \mathcal{D}_{n, z_1'}, &\varphi \> - \< \mathcal{D}_{n, z_2'}, \varphi \> | \\
 &\le \int dx_1 \hdots \int dx_n \left| \splitalign{\psi(x_1 \hdots x_n z_1') e^{- W(0^\infty | x_1 \hdots x_n z_1')}}{\hspace{3em} - \psi(x_1 \hdots x_n z_2') e^{- W(0^\infty | x_1 \hdots x_n z_2')}} \right| \varphi(0^\infty x_1 \hdots x_n) \frac{e^{\Re (A^\star)^n(0^\infty x_1 \hdots x_n)}}{|\lambda|^n} \\
 &\le \left\| \psi(\cdot) e^{-W(0^\infty | \cdot)} \right\|_\theta \( \frac{\theta}{|\lambda|} \)^n \| \varphi \|_\infty \int dx_1 \hdots \int dx_n e^{\Re (A^\star)^n(0^\infty x_1 \hdots x_n)} \\
 &\le \left\| \psi(\cdot) e^{-W(0^\infty | \cdot)} \right\|_\theta \( \frac{\theta}{|\lambda|} \)^n \| \varphi \|_\infty (\rho + \varepsilon)^n
\end{align*}
and this independently of our choice of $z_1'$ and $z_2'$.
\end{proof}

We also need to control the speed of convergence of the sequence $(\mathcal{D}_{n, x'})_{n \ge 0}$ uniformly in $x' \in \Omega$.

\begin{lemma} \label{ooo}
If $\varphi \in H_\theta(\Omega)$ then for every $n \ge n_0$ and $x' \in \Omega$ we have :
\begin{equation} \label{in}
| \< \mathcal{D}_{n+1, x'}, \varphi \> - \< \mathcal{D}_{n, x'}, \varphi \> | \le \left\| e^{W(\cdot | x')} \varphi(\cdot) \right\|_\theta \| \psi \|_\infty \frac{\rho + \varepsilon}{|\lambda|} \( \frac{(\rho + \varepsilon) \theta}{|\lambda|} \)^n
\end{equation}
\end{lemma}

\begin{proof}
Like in lemma \ref{oo}, $y' \mapsto e^{W(y' | x')} \varphi(y')$ is $\theta$-H\"older as the product of two $\theta$-H\"older maps.
Using that $\psi$ is an $\lambda$-eigenfuntion of $\mathcal{L}_A$ and the definition of the Birkhoff sum of $A$, we get :
\begin{align*}
 \< \mathcal{D}_{n, x'}, \varphi \> &= \int dx_1 \hdots \int dx_n \left[ \frac{1}{\lambda} \int \psi(a x_n \hdots x_1 x') e^{A(a x_n \hdots x_1 x')} da \right] \\
 &\hspace{5em} \frac{e^{A^n(x_n \hdots x_1 x')}}{\lambda^n} e^{- W(0^\infty x_n \hdots x_1 | x')} \varphi(0^\infty x_n \hdots x_1) \\
 &= \int dx_1 \hdots \int dx_{n+1} \psi(x_{n+1} \hdots x_1 x') \frac{e^{A^{n+1}(x_{n+1} \hdots x_1 x')}}{\lambda^{n+1}} \\
 &\hspace{7em} e^{- W(0^\infty x_n \hdots x_1 | x')} \varphi(0^\infty x_n \hdots x_1)
\end{align*}
Hence, using that $0^\infty x_{n+1} x_n \hdots x_1 \sim_n 0^\infty 0 x_n \hdots x_1$ in $\Omega^\star$ for every $x_1, \hdots, x_n \in M$, and the definition of $\mathcal{D}_{n+1, x'}$ we have :
\begin{align*}
 | \< \mathcal{D}_{n, x'}, \varphi \> - \< \mathcal{D}_{n+1, x'}, \varphi \> | &\le \int dx_1 \hdots \int dx_{n+1} |\psi(x_{n+1} \hdots x_1 x')| \frac{e^{\Re A^{n+1}(x_{n+1} \hdots x_1 x')}}{|\lambda|^{n+1}} \\
 &\hspace{4em} \left| \splitalign{e^{- W(0^\infty x_{n+1} x_n \hdots x_1 | x')} \varphi(0^\infty x_{n+1} x_n \hdots x_1)}{\hspace{7em} - e^{- W(0^\infty 0 x_n \hdots x_1 | x')} \varphi(0^\infty 0 x_n \hdots x_1)} \right| \\
 &\le \| \psi \|_\infty \frac{(\rho + \varepsilon)^{n+1}}{|\lambda|^{n+1}} \left\| e^{W(\cdot | x')} \varphi(\cdot) \right\|_\theta \theta^n \tag*{\qedhere}
\end{align*}
\end{proof}

We can now combine lemmas \ref{oo} and \ref{ooo} together to show that the sequence $(\mathcal{D}_{n, x'})_{n \ge 0}$ converges and that its limit does not depend on the base point $x'$.
This limit will be our candidate for the preimage of $\psi$ by the map $\Phi_W$, in the sense of theorem \ref{th:ruelle}.

\begin{lemma} \label{defd}
Recall that $|\lambda| > (\rho + \varepsilon) \theta$.
If $\psi \in H_\theta(\Omega)$ is a $\lambda$-eigenfunction of $\mathcal{L}_A$, then for every $\varphi \in H_\theta(\Omega^\star)$ and for every $x' \in \Omega$ the limit $\lim_{n \to \infty} \< \mathcal{D}_{n, x'}, \varphi \>$ exists and does not depend on $x'$.
This defines a linear functional $\mathcal{D} : H_\theta(\Omega^\star) \to \mathbb{C}$ in such way that :
\[ \< \mathcal{D}, \varphi \> = \lim_{n \to \infty} \< \mathcal{D}_{n, x'}, \varphi \> \]
Moreover, for every $x' \in \Omega$ and $n \ge n_0$ we have the estimate :
\[ | \< \mathcal{D}, \varphi \> - \< \mathcal{D}_{n, x'}, \varphi \> | \le K_{x'}(\varphi) \( \frac{(\rho + \varepsilon) \theta}{|\lambda|} \)^n \]
where $K_{x'}(\varphi) = \left\| e^{W(\cdot | x')} \varphi(\cdot) \right\|_\theta \| \psi \|_\infty \displaystyle\frac{\rho + \varepsilon}{|\lambda| - (\rho + \varepsilon) \theta}$.
\end{lemma}

\begin{proof}
We fix $x' \in \Omega$.
Let $K_{x'}'(\varphi) = \left\| e^{W(\cdot | x')} \varphi(\cdot) \right\|_\theta \| \psi \|_\infty \frac{\rho + \varepsilon}{|\lambda|}$.
Iterating the inequality \eqref{in} from lemma \ref{ooo}, we get that for any $k \ge 0$ and $n \ge n_0$, using a telescopic sum :
\begin{align*}
 | \< \mathcal{D}_{n+k, x'}, \varphi \> - \< \mathcal{D}_{n, x'}, \varphi \> | &\le K_{x'}'(\varphi) \sum_{j=0}^{k-1} \( \frac{(\rho + \varepsilon) \theta}{|\lambda|} \)^{n+j} \\
 &\le K_{x'}'(\varphi) \( \frac{(\rho + \varepsilon) \theta}{|\lambda|} \)^n \frac{1}{1 - \frac{(\rho + \varepsilon) \theta}{|\lambda|}}
\end{align*}
This expression shows that $(\< \mathcal{D}_{n, x'}, \varphi \>)_{n \ge 0}$ is a Cauchy sequence, hence converges when $n$ goes to infinity.
We denote by $\mathcal{D}_{x'}$ the linear functional such that :
\[ \< \mathcal{D}_{x'}, \varphi \> = \lim_{n \to \infty} \< \mathcal{D}_{n, x'}, \varphi \> \]
Then, for every $n \ge n_0$ :
\begin{equation} \label{distconv}
 | \< \mathcal{D}_{x'}, \varphi \> - \< \mathcal{D}_{n, x'}, \varphi \> | \le K_{x'}(\varphi) \( \frac{(\rho + \varepsilon) \theta}{|\lambda|} \)^n
\end{equation}
where $K_{x'}(\varphi) = K_{x'}'(\varphi) \frac{|\lambda|}{|\lambda| - (\rho + \varepsilon) \theta}$.

Now, by a $3 \varepsilon$ argument, if $a, b \in \Omega$, lemma \ref{oo} ensures that for every $n \ge n_0$ and every map $\varphi \in H_\theta(\Omega)$ :
\begin{align*}
 | \< \mathcal{D}_{a}, \varphi \> &- \< \mathcal{D}_{b}, \varphi \> | \\
 &\le | \< \mathcal{D}_{a}, \varphi \> - \< \mathcal{D}_{a, n}, \varphi \> | + | \< \mathcal{D}_{a, n}, \varphi \> - \< \mathcal{D}_{b, n}, \varphi \> | + \< \mathcal{D}_{b, n}, \varphi \> - \< \mathcal{D}_{b}, \varphi \> | \\
 &\le K_{a}(\varphi) \( \frac{(\rho + \varepsilon) \theta}{|\lambda|} \)^n + C'(\varphi) \( \frac{(\rho + \varepsilon) \theta}{|\lambda|} \)^n + K_{b}(\varphi) \( \frac{(\rho + \varepsilon) \theta}{|\lambda|} \)^n
\end{align*}
where $C'(\varphi) = \left\| \psi(\cdot) e^{-W(0^\infty | \cdot)} \right\|_\theta \| \varphi \|_\infty$.
But since $W$ (hence $e^W$) is $\theta$-H\"older, we have an upper bound for :
\[ K_{x'}(\varphi) \le \( \left\| e^W \right\|_\theta \| \varphi \|_\infty + \left\| e^W \right\|_\infty \| \varphi \|_\theta \) \| \psi \|_\infty \frac{|\lambda|}{|\lambda| - (\rho + \varepsilon) \theta} \]
which is independent of $x'$.
This implies that this expression goes to zero when $n$ goes to infinity independently of $x'$, and :
\[ \< \mathcal{D}, \varphi \> = \< \mathcal{D}_{x'}, \varphi \> \]
is well-defined for every map $\varphi$.
The estimate follows then immediately from \eqref{distconv}.
\end{proof}

For $\mathcal{D}$ to be a preimage of $\psi$ in the sense of theorem \ref{th:ruelle}, it first needs to be a continuous linear functional on $H_\theta(\Omega^\star)$.

\begin{lemma}
The linear functional $\mathcal{D} : H_\theta(\Omega^\star) \to \mathbb{C}$ is continuous with respect to the norm $\| \cdot \|_\infty + \| \cdot \|_\theta$.
\end{lemma}

\begin{proof}
Let $\varphi \in H_\theta(\Omega^\star)$, and fix some $x' \in \Omega$.
According to the estimate from lemma \ref{defd} specialized for $n = n_0$, we have :
\begin{align*}
 | \< \mathcal{D}, \varphi \> - \< \mathcal{D}_{n_0, x'}, \varphi \> | &\le \left\| e^{W(\cdot | x')} \varphi(\cdot) \right\|_\theta \| \psi \|_\infty \frac{\rho + \varepsilon}{|\lambda| - (\rho + \varepsilon) \theta} \( \frac{(\rho + \varepsilon) \theta}{|\lambda|} \)^{n_0} \\
 &\le L_1 \| \varphi \|_\infty + L_2 \| \varphi \|_\theta
\end{align*}
where :
\begin{align*}
 L_1 &= \left\| e^{W(\cdot | x')} \right\|_\infty \| \psi \|_\infty \frac{\rho + \varepsilon}{|\lambda| - (\rho + \varepsilon) \theta} \( \frac{(\rho + \varepsilon) \theta}{|\lambda|} \)^{n_0} \\
 L_2 &= \left\| e^{W(\cdot | x')} \right\|_\theta \| \psi \|_\infty \frac{\rho + \varepsilon}{|\lambda| - (\rho + \varepsilon) \theta} \( \frac{(\rho + \varepsilon) \theta}{|\lambda|} \)^{n_0}
\end{align*}
But since lemma \ref{dncont} gives for $n = n_0$ that :
\[ | \< \mathcal{D}_{n_0, x'}, \varphi \> | \le \| \psi \|_\infty \left\| e^{-W} \right\|_\infty \( \frac{\rho + \varepsilon}{|\lambda|} \)^{n_0} \| \varphi \|_\infty \]
we get that $| \< \mathcal{D}, \varphi \> | \le L_1' \| \varphi \|_\infty + L_2 \| \varphi \|_\theta$ where :
\[ L_1' = L_1 + \| \psi \|_\infty \left\| e^{-W} \right\|_\infty \( \frac{\rho + \varepsilon}{|\lambda|} \)^{n_0} \]
This completes the proof.
\end{proof}

We now have to check that this candidate is indeed an eigendistribution of $\mathcal{L}_{A^\star}$ for the eigenvalue $\lambda$.

\begin{lemma}
For any $\varphi \in H_\theta(\Omega^\star)$ we have :
\[ \< \mathcal{D}, \mathcal{L}_{A^\star} \varphi \> = \lambda \< \mathcal{D}, \varphi \> \]
\end{lemma}

\begin{proof}
Let $\varphi \in H_\theta(\Omega^\star)$, and fix $n \ge n_0$.
First note that lemma \ref{int} ensures that :
\[ \lambda \< \mathcal{D}_{n+1, x'}, \varphi \> = \int \left\< \mathcal{D}_{n, a x'}, e^{A^\star(\cdot a)} \varphi(\cdot a) \right\> da \]
hence we get that :
\begin{align*}
 | \lambda \< \mathcal{D}, \varphi \> &- \< \mathcal{D}, \mathcal{L}_{A^\star} \varphi \> | \le \left| \lambda \< \mathcal{D}, \varphi \> - \lambda \< \mathcal{D}_{n+1, x'}, \varphi \> \right| \\
 &\hspace{6em}+   \left| \lambda \< \mathcal{D}_{n+1, x'}, \varphi \> - \int \left\< \mathcal{D}_{n, a x'}, e^{A^\star(\cdot a)} \varphi(\cdot a) \right\> da \right| \\
 &\hspace{6em}+   \left| \int \left\< \mathcal{D}_{n, a x'}, e^{A^\star(\cdot a)} \varphi(\cdot a) \right\> da - \< \mathcal{D}, \mathcal{L}_{A^\star} \varphi \> \right| \\
 &\le \left| \lambda \< \mathcal{D}, \varphi \> - \lambda \< \mathcal{D}_{n+1, x'}, \varphi \> \right| + \left| \int \left\< \mathcal{D}_{n, a x'}, e^{A^\star(\cdot a)} \varphi(\cdot a) \right\> da - \< \mathcal{D}, \mathcal{L}_{A^\star} \varphi \> \right|
\end{align*}
We will estimate each part of this upper bound successively.

We start by the first term.
Since $n \ge n_0$, the estimate of lemma \ref{defd} gives that :
\[ \left| \lambda \< \mathcal{D}, \varphi \> - \lambda \< \mathcal{D}_{n+1, x'}, \varphi \> \right| \le |\lambda| K_{x'}(\varphi) \( \frac{(\rho + \varepsilon) \theta}{|\lambda|} \)^{n+1} \]
This will be enough for our needs.

Now focus on the second term.
We first note that by definition of $\mathcal{L}_{A^\star}$ :
\begin{equation} \label{eq:invsecondpart}
 \left| \int \left\< \mathcal{D}_{n, a x'}, e^{A^\star(\cdot a)} \varphi(\cdot a) \right\> da - \< \mathcal{D}, \mathcal{L}_{A^\star} \varphi \> \right| \le \int \left| \left\< \mathcal{D}_{n, a x'}, \tilde{\varphi}_a \right\> - \left\< \mathcal{D}, \tilde{\varphi}_a \right\> \right| da
\end{equation}
where $\tilde{\varphi}_a(\cdot) = e^{A^\star(\cdot a)} \varphi(\cdot a)$.
Then, the estimate of lemma \ref{defd} specialized for every $\tilde{\varphi}_a, a \in M$ gives :
\[ \left| \< \mathcal{D}_{n, a x'}, \tilde{\varphi}_a \> - \< \mathcal{D}, \tilde{\varphi}_a \> \right| \le K_{a x'}(\tilde{\varphi}_a) \( \frac{(\rho + \varepsilon) \theta}{|\lambda|} \)^n \]
We now need to find an upper bound for $K_{a x'}(\tilde{\varphi}_a)$ that does not depend on $a$.
From the definition of $K_{a x'}(\tilde{\varphi}_a)$, we have :
\[ K_{a x'}(\tilde{\varphi}_a) \le \( \left\| e^{W(\cdot | x')} \right\|_\theta \| \tilde{\varphi}_a \|_\infty + \left\| e^{W(\cdot | x')} \right\|_\infty \| \tilde{\varphi}_a \|_\theta \) \frac{\rho + \varepsilon}{|\lambda| - (\rho + \varepsilon) \theta} \]
hence it is enough to find an upper bound of $\| \tilde{\varphi}_a \|_\infty$ and $\| \tilde{\varphi}_a \|_\theta$ that does not depend on $a$.
It is clear that $\| \tilde{\varphi}_a \|_\infty \le \| \varphi \|_\infty$.
About $\| \tilde{\varphi}_a \|_\theta$, note that for every $a \in M$ :
\[ \| \tilde{\varphi}_a \|_\theta \le \left\| e^{A^\star(\cdot a)} \right\|_\theta \| \varphi(\cdot a) \|_\infty + \| \varphi(\cdot a) \|_\theta \left\| e^{A^\star(\cdot a)} \right\|_\infty \]
We need estimates for both $\| e^{A^\star(\cdot a)} \|_\theta$ and $\| \varphi(\cdot a) \|_\theta$.
Take $z, z' \in \Omega^\star$ such that $z \sim_p z'$.
Then $z a \sim_{p+1} z' a$ and we get :
\[ |\varphi(z a) - \varphi(z' a)| \le \| \varphi \|_\theta \theta^{p+1} \]
which shows that $\| \varphi(. a) \|_\theta \le \| \varphi \|_\theta \theta$ and likewise $\| e^{A^\star(. a)} \|_\theta \le \| e^{A^\star} \|_\theta \theta$.
Thus :
\[ \| \tilde{\varphi}_a \|_\theta \le \theta \( \left\| e^{A^\star} \right\|_\theta \| \varphi \|_\infty + \| \varphi \|_\theta \left\| e^{A^\star} \right\|_\infty \) \]
independently of $a$.
Therefore, there exists a constant $\tilde{K}_{x'}(\varphi)$ such that, for every $a \in M$ :
\[ \left| \left\< \mathcal{D}_{n, a x'}, \tilde{\varphi}_a \right\> - \left\< \mathcal{D}_{a x'}, \tilde{\varphi}_a \right\> \right| \le \tilde{K}_{x'}(\varphi) \( \frac{(\rho + \varepsilon) \theta}{|\lambda|} \)^n \]
Then, by integrating this inequality and plugging it into \eqref{eq:invsecondpart}, we get :
\begin{align*}
 &\left| \int \left\< \mathcal{D}_{n, a x'}, e^{A^\star(\cdot a)} \varphi(\cdot a) \right\> da - \< \mathcal{D}, \mathcal{L}_{A^\star} \varphi \> \right| \\
 &\hspace{5em} \le \int \tilde{K}_{x'}(\varphi) \( \frac{(\rho + \varepsilon) \theta}{|\lambda|} \)^n da = \tilde{K}_{x'}(\varphi) \( \frac{(\rho + \varepsilon)}{|\lambda|} \)^n
\end{align*}

Gathering these two estimates, we finally get that :
\[ \left| \lambda \< \mathcal{D}, \varphi \> - \< \mathcal{D}, \mathcal{L}_{A^\star} \varphi \> \right| \le |\lambda| K_{x'}(\varphi) \( \frac{(\rho + \varepsilon) \theta}{|\lambda|} \)^{n+1} + \tilde{K}_{x'}(\varphi) \( \frac{(\rho + \varepsilon) \theta}{|\lambda|} \)^n \]
which goes to zero when $n$ goes to the infinty.
\end{proof}

Finally, we need to check that $\mathcal{D}$ is actually a preimage of $\psi$ by $\Phi_W$.

\begin{lemma} \label{lem:surj}
$\Phi_W(\mathcal{D}) = \psi$.
\end{lemma}

\begin{proof}
For a fixed $x'$ and any $n \ge 0$, we have :
\begin{align*}
 \left\< \mathcal{D}_{n, x'}, e^{W(\cdot | x')} \right\> &= \int dx_1 \hdots \int dx_n \psi(x_n \hdots x_1 x') \\
 &\hspace{6em} \frac{e^{A^n(x_n \hdots x_1 x')}}{\lambda^n} e^{- W(0^\infty x_n \hdots x_1 | x')} e^{W(0^\infty x_n \hdots x_1 | x')} \\
 &= \frac{1}{\lambda^n} \mathcal{L}^n_A \psi(x') = \psi(x')
\end{align*}
Thus, for any $n \ge n_0$ and $x' \in \Omega$, the estimate of lemma \ref{defd} gives :
\begin{align*}
 \left| \left\< \mathcal{D}, e^{W(\cdot | x')} \right\> - \psi(x') \right| &= \left| \left\< \mathcal{D}_{x'}, e^{W(\cdot | x')} \right\> - \left\< \mathcal{D}_{n, x'}, e^{W(\cdot | x')} \right\> \right| \\
 &\le K_{x'}\( e^{W(\cdot | x')} \) \( \frac{(\rho + \varepsilon) \theta}{|\lambda|} \)^n
\end{align*}
which converges to $0$ when $n$ goes to the infinity.
\end{proof}

Note that in \cite{Pit1}, \cite{Pit3} and \cite{LT}, where similar results are obtained for a specific family of Markov maps of the circle, it was possible to prove directly that the analogue of the $\Phi_W$ map is injective thanks to the smooth structure on the circle which gives additional regularity to the eigendistributions.
This is not the case in our symbolic setting, so we must resort to a dimension argument.

Denote by :
\begin{align*}
 E_\lambda(A)       &= \left\{ \psi \in H_\theta(\Omega) \,\middle|\, \mathcal{L}_A \psi = \lambda \psi \right\} \\
 F_\lambda(A^\star) &= \left\{ \nu \in H_\theta(\Omega^\star)' \,\middle|\, \forall \varphi \in H_\theta(\Omega^\star), \< \nu, \mathcal{L}_{A^\star} \varphi \> = \lambda \< \nu, \varphi \> \right\}
\end{align*}
the spaces of respectively $\lambda$-eigenfunctions of $\mathcal{L}_A$ and $\lambda$-eigendistributions of $\mathcal{L}_{A^\star}$.
The spaces $E_\lambda(A^\star)$ and $F_\lambda(A)$ are defined in a similar way.
Note that $F_\lambda(A^\star)$ (respectively $F_\lambda(A)$) formally coincides with the $\lambda$-eigenspace of the dual operator $\mathcal{L}_{A^\star}^\star$ of $\mathcal{L}_{A^\star}$ (respectively $\mathcal{L}_A^\star$ of $\mathcal{L}_A$).
When $\lambda$ is in the isolated spectrum, we shall prove that these vector spaces $E_\lambda(A)$ and $F_\lambda(A^\star)$ have the same finite dimension.

\begin{lemma} \label{lem:dim}
If $|\lambda| > \rho \theta$, then $\dim E_\lambda(A) = \dim F_\lambda(A^\star)$.
\end{lemma}

\begin{proof}
If $|\lambda| > \rho \theta$, then $\lambda$ lies in the isolated spectrum of $\mathcal{L}_A : H_\theta(\Omega) \to H_\theta(\Omega)$, whose essential spectral radius is bounded from above by $\rho \theta$ according to \cite{PP}.
Given that $\mathcal{L}_A$ is a bounded operator, lemma VIII.8.2 from \cite{DS} ensures that the $\lambda$-eigenspace $E_\lambda(A)$ and the $\lambda$-eigenspace of its dual operator $\mathcal{L}_A^\star$ have same finite dimension, hence :
\[ \dim E_\lambda(A) = \dim F_\lambda(A) \]
Moreover, $|\lambda| > \rho \theta = \rho^\star \theta$ according to lemma $\ref{rho}$, so it is also in the isolated spectrum of $\mathcal{L}_{A^\star} : H_\theta(\Omega^\star) \to H_\theta(\Omega^\star)$ and the same argument gives :
\[ \dim E_\lambda(A^\star) = \dim F_\lambda(A^\star) \]
Since $\Phi_W$ is surjective from $F_\lambda(A^\star)$ to $E_\lambda(A)$, we already know that :
\[ \dim E_\lambda(A) \le \dim F_\lambda(A^\star) \]
But note that the equation \eqref{W0} that defines the involution kernel $W$ can be rewritten after composition by $\hat{\sigma}$ as :
\[ A(x) = A^\star(\hat{\sigma}(y | x)) + W(\hat{\sigma}(y | x)) - W(y | x) \]
which means that $W$ is also an involution kernel between $A$ and $A^\star$ for the inverse shift $(X, \hat{\sigma}^{-1})$.
The whole construction can then be applied to this new setup where the roles of $(\Omega, \sigma, A)$ and $(\Omega^\star, \sigma^\star, A^\star)$ are reversed, and we get that the corresponding $\Phi_W : H_\theta(\Omega)' \to \mathcal{C}(\Omega^\star)$ is surjective from $F_\lambda(A)$ to $E_\lambda(A^\star)$.
Hence :
\[ \dim E_\lambda(A^\star) \le \dim F_\lambda(A) \]
This implies that all these vector spaces have the same dimension.
\end{proof}

Since lemma \ref{lem:surj} shows that the map $\Phi_W$ is surjective from $F_\lambda(A^\star)$ to $E_\lambda(A)$, and that lemma \ref{lem:dim} tells us that these two vector spaces have the same dimension, we have completed the proof of theorem \ref{th:ruelle}.

\section{Koopman operator duality}

In this section, we will give a proof of theorem \ref{th:koopman}, which establishes a relation between eigenfunctions of $U_B$ and eigendistributions of $U_{C^\star}$.
Denote by :
\begin{align*}
 \tilde{E}_\lambda(B)       &= \left\{ \psi \in H_\theta(\Omega) \,\middle|\, U_B \psi = \lambda \psi \right\} \\
 \tilde{F}_\lambda(C^\star) &= \left\{ \nu \in H_\theta(\Omega^\star)' \,\middle|\, \forall \varphi \in H_\theta(\Omega^\star), \< \nu, U_{C^\star} \varphi \> = \lambda \< \nu, \varphi \> \right\}
\end{align*}
the $\lambda$-eigenspaces of respectively $U_B$ and $U_{C^\star}$.

We shall make extensive use of this fundamental relation between a Ruelle operator $\mathcal{L}_A : H_\theta(\Omega) \to H_\theta(\Omega)$ and a Koopman operator $U_B : H_\theta(\Omega) \to H_\theta(\Omega)$ for any choice of $A, B \in H_\theta(\Omega)$ :
\begin{equation} \label{eq:trans-koop}
 \forall f, g \in H_\theta(\Omega), \mathcal{L}_A \left[ f U_B g \right] = g \mathcal{L}_A \left[ e^B f \right]
\end{equation}

It can be used to derive an isomorphism between the spaces of eigenfunctions of $\mathcal{L}_A$ and $U_B$ for any potentials $A$ and $B$.

\begin{proposition} \label{prop:trans-koop-func}
 For every $A, B \in H_\theta(\Omega)$, there exist a map $f_{A+B} \in H_\theta(\Omega)$ with $f_{A+B} > 0$ and a $\lambda_{A+B} > 0$ such that for every $\lambda \in \mathbb{C} \setminus \{ 0 \}$ the map :
 \[ \mathcal{E}_{A+B} : \psi \in \tilde{E}_{\frac{\lambda_{A+B}}{\lambda}}(B) \mapsto f_{A+B} \psi \in E_\lambda(A) \]
 is an isomorphism.
\end{proposition}

\begin{proof}
 Since $A + B \in H_\theta(\Omega)$, by Ruelle-Perron-Frobenius' theorem, we have the existence of $f_{A+B} \in H_\theta(\Omega), f_{A+B} > 0$ and $\lambda_{A+B} > 0$ such that :
 \[ \mathcal{L}_{A+B} f_{A+B} = \lambda_{A+B} f_{A+B} \]
 We shall prove that these $f_{A+B}$ and $\lambda_{A+B}$ are suitable.
 First, it is clear that if $\psi \in H_\theta(\Omega)$ then so does $f_{A+B} \psi$.
 Now assume that $U_B \psi = \frac{\lambda_{A+B}}{\lambda} \psi$.
 Using equation \eqref{eq:trans-koop}, we have :
 \[ \psi \mathcal{L}_A \left[ e^B f_{A+B} \right] = \mathcal{L}_A \left[ f_{A+B} U_B \psi \right] = \frac{\lambda_{A+B}}{\lambda} \mathcal{L}_A \left[ f_{A+B} \psi \right] \]
 But by definition of $f_{A+B}$ we also have :
 \[ \mathcal{L}_A \left[ e^B f_{A+B} \right] = \mathcal{L}_{A+B} f_{A+B} = \lambda_{A+B} f_{A+B} \]
 Hence :
 \[ \lambda_{A+B} f_{A+B} \psi = \frac{\lambda_{A+B}}{\lambda} \mathcal{L}_A \left[ f_{A+B} \psi \right] \]
 which shows that $f_{A+B} \psi \in E_\lambda(A)$.
 Finally, since $f_{A+B} > 0$, it is clear that $\mathcal{E}_{A+B}$ is an isomorphism.
\end{proof}

We also have a similar isomorphism between the spaces of eigendistributions of $\mathcal{L}_{A^\star}$ and $U_{C^\star}$.

\begin{proposition} \label{prop:trans-koop-dist}
 For every $A^\star, C^\star \in H_\theta(\Omega^\star)$, there exist a map $f_{A^\star+C^\star} \in H_\theta(\Omega^\star)$ with $f_{A^\star+C^\star} > 0$ and a $\lambda_{A^\star+C^\star} > 0$ such that for every $\lambda \in \mathbb{C} \setminus \{ 0 \}$ the map :
 \[ \mathcal{F}_{A^\star+C^\star} : \nu \in F_\lambda(A^\star) \mapsto f_{A^\star+C^\star} \nu \in \tilde{F}_{\frac{\lambda_{A^\star+C^\star}}{\lambda}}(C^\star) \]
 is an isomorphism.
\end{proposition}

\begin{proof}
 Since $A^\star + C^\star \in H_\theta(\Omega^\star)$, by Ruelle-Perron-Frobenius' theorem, we have the existence of $f_{A^\star+C^\star} \in H_\theta(\Omega^\star), f_{A^\star+C^\star} > 0$ and $\lambda_{A^\star+C^\star} > 0$ such that :
 \[ \mathcal{L}_{A^\star+C^\star} f_{A^\star+C^\star} = \lambda_{A^\star+C^\star} f_{A^\star+C^\star} \]
 We shall prove that these $f_{A^\star+C^\star}$ and $\lambda_{A^\star+C^\star}$ are suitable.
 First, it is clear that if $\nu \in H_\theta(\Omega^\star)'$ then so does $f_{A^\star+C^\star} \nu$.
 Now assume that $\mathcal{L}_{A^\star}^\star \nu = \lambda \nu$.
 Using equation \eqref{eq:trans-koop}, we have :
 \begin{align*}
  \forall \varphi \in H_\theta(\Omega^\star), \lambda \< \nu, f_{A^\star+C^\star} U_{C^\star} \varphi \> &= \left\< \nu, \mathcal{L}_{A^\star} \left[ f_{A^\star+C^\star} U_{C^\star} \varphi \right] \right\> \\
  &= \left\< \nu, \varphi \mathcal{L}_{A^\star} \left[ e^{C^\star} f_{A^\star+C^\star} \right] \right\>
 \end{align*}
 But by definition of $f_{A^\star+C^\star}$ we also have :
 \[ \mathcal{L}_{A^\star} \left[ e^{C^\star} f_{A^\star+C^\star} \right] = \mathcal{L}_{A^\star+C^\star} f_{A^\star+C^\star} = \lambda_{A^\star+C^\star} f_{A^\star+C^\star} \]
 Hence :
 \[ \forall \varphi \in H_\theta(\Omega^\star), \lambda \< f_{A^\star+C^\star} \nu, U_{C^\star} \varphi \> = \lambda_{A^\star+C^\star} \< f_{A^\star+C^\star} \nu, \varphi \> \]
 which shows that $f_{A^\star+C^\star} \nu \in \tilde{F}_{\frac{\lambda_{A^\star+C^\star}}{\lambda}}(C^\star)$.
 Finally, since $f_{A^\star+C^\star} > 0$, it is clear that $\mathcal{F}_{A^\star+C^\star}$ is an isomorphism.
\end{proof}

We can now combine these two propositions with the results from the previous section to get the proof of theorem \ref{th:koopman}.
Indeed, it is enough to take $f = f_{A+B}$, $\alpha = \lambda_{A+B}$, $g = f_{A^\star+C^\star}$, and $\beta = \lambda_{A^\star+C^\star}$ as given by propositions \ref{prop:trans-koop-func} and \ref{prop:trans-koop-dist} to obtain that :
\[ \Psi_{A,B,C^\star} = \mathcal{E}_{A+B} \Phi_W \mathcal{F}_{A^\star+C^\star}^{-1} \]
In particular, its inverse is :
\[ \Psi_{A,B,C^\star}^{-1} = \mathcal{F}_{A^\star+C^\star} \Phi_W^{-1} \mathcal{E}_{A+B}^{-1} \]
The result follows then immediately from the expression of $\Phi_W^{-1}$ in theorem \ref{th:ruelle}.


\begin{thebibliography}{}

\bibitem{ABL}
A. B. Antonevich, V. I. Bakhtin and A. V. Lebedev,
On t-entropy and variational principle for the spectral radii of transfer and weighted shift operators,
{\it Erg. Th. and Dyn. Sys.} {\bf 31} (2011), Number 4, 995--1042.

\bibitem{BCLMS}
A. T. Baraviera, L. Cioletti, A. O. Lopes, J. Mohr and R. R. Souza,
On the general one-dimensional XY model: positive and zero temperature, selection and non-selection,
{\it Rev. Math. Phys.} {\bf 23} (2011), Number 10, 1063--1113.

\bibitem{BKL02}
M. Blank, G. Keller and C. Liverani,
Ruelle-Perron-Frobenius spectrum for Anosov maps,
{\it Nonlinearity} {\bf 15} (2002), Number 6, 1905--1973.

\bibitem{BLLbook}
A. Baraviera, R. Leplaideur and A. O. Lopes,
Ergodic Optimization, Zero temperature limits and the Max-Plus Algebra,
Mini-course in XXIX Col\'oquio Brasileiro de Matem\'atica, IMPA - Rio de Janeiro (2013)

\bibitem{BLT}
A. Baraviera, A. O. Lopes and P. Thieullen,
A large deviation principle for equilibrium states of H\"older potencials: the zero temperature case,
{\it Stochastics and Dynamics} {\bf 6} (2006), 77--96.

\bibitem{BT08}
V. Baladi and M. Tsujii,
Anisotropic H\"older and Sobolev spaces for hyperbolic diffeomorphisms,
{\it Ann. Inst. Fourier} {\bf 57} (2008), Number 1, 127--154.

\bibitem{Bal}
V. Baladi,
Positive Transfer Operators and Decay of Correlations,
{\it Advanced Series in Nonlinear Dynamics} {\bf 16} (2000).

\bibitem{Bo1}
E.B. Bogomolny and M. Carioli,
Quantum maps of geodesic flows on surfaces of constant negative curvature,
{\it IV International Conference on ``Path Integrals from meV to MeV'', Tutzing May} (1992), 18--21.

\bibitem{Bo2}
E.B. Bogomolny and M. Carioli,
Quantum maps from transfer operator,
{\it Physica D} {\bf 67} (1993), 88--112.

\bibitem{DS}
N. Dunford and J. T. Schwartz,
Linear Operators, Part I : General Theory,
{\it Interscience Publishers}, 1958.

\bibitem{Fro}
G. Froyland and O. Stancevic,
Escape rates and Perron-Frobenius operators: open and closed dynamical systems,
{\it Disc. Cont. Dyn. Syst. Ser. B} {\bf 14} (2010), Number 2, 457--472.

\bibitem{GLP13}
P. Giulietti, C. Liverani and M. Pollicott,
Anosov flows and dynamical zeta functions,
{\it Ann. of Math. (2)} {\bf 178} (2013), Number 2, 687--773.

\bibitem{KR}
G. Keller and H. H. Rugh,
Eigenfunctions for smooth expanding circle maps,
{\it Nonlinearity} {\bf 17} (2014), Number 5, 1723--1730.

\bibitem{LT}
A. O. Lopes and P. Thieullen,
Eigenfunctions of the Laplacian and associated Ruelle operator,
{\it Nonlinearity}, {\bf 21} (2008), Number 10, 2239--2254.

\bibitem{LMMS}
A. O. Lopes, J. K. Mengue, J. Mohr and R. R. Souza,
Entropy and Variational Principle for one-dimensional Lattice Systems with a general a-priori probability: positive and zero temperature,
to appear in {\it Erg. Th. and Dyn. Syst.}.

\bibitem{LOT}
A. O. Lopes, E. R. Oliveira and Ph. Thieullen,
The dual potential, the involution kernel and transport in ergodic optimization,
to appear {\it Dynamics, Games and Science}, Edit. J-P Bourguignon, R.
Jelstch, A. Pinto and M. Viana, Springer Verlag, 331-364.

\bibitem{PP}
W. Parry and M. Pollicott,
Zeta functions and the periodic orbit structure of hyperbolic dynamics,
{\it Ast\'erisque} {\bf 187--188} (1990).

\bibitem{Pit1}
V. Pit,
Codage du flot g\'eodesique sur les surfaces hyperboliques de volume fini,
Th\`ese de doctorat, Bordeaux (2010).

\bibitem{Pit2}
V. Pit,
Invariant relations for the Bowen-Series transform,
{\it Conformal Geometry and Dynamics} {\bf 16} (2012), 103--123.

\bibitem{Pit3}
V. Pit,
Ruelle operator duality for coupled smooth maps of the circle,
{\it preprint} (2013).

\bibitem{Ruelle}
D. Ruelle,
Thermodynamic Formalism,
{\it Cambridge} (2004), second edition.

\end{thebibliography}
\end{document}